\documentclass[12pt]{amsart}
\usepackage{amssymb,amsthm,mathrsfs, mathdots,bm,cite}
\usepackage[centertags]{amsmath}

\newtheorem{thm}{Theorem}[section]
\newtheorem{theorem}[thm]{Theorem}
\newtheorem{corollary}[thm]{Corollary}
\newtheorem{lemma}[thm]{Lemma}
\newtheorem{proposition}[thm]{Proposition}
\newtheorem{example}[thm]{Example}

\theoremstyle{definition}

\theoremstyle{remark}
\newtheorem{remark}[thm]{Remark}

\newenvironment{theorem*}[1]{\smallskip\noindent{\bf #1.}\it}{\medskip}

\numberwithin{equation}{section} 
\newcommand\dom{\operatorname{dom}}

\newcommand\bC{{\mathbb C}}

\newcommand\bR{{\mathbb R}}
\newcommand\bZ{{\mathbb Z}}

\newcommand\cL{{\mathscr L}}

\newcommand\cA{{\mathcal A}}
\newcommand\cE{{\mathcal E}}

\newcommand\al{\alpha}
\newcommand\la{\lambda}

\DeclareMathOperator*{\myres}{\mathrm{res}}

\begin{document}

\title{Spectral properties of {S}turm--{L}iouville equations with singular energy-dependent potentials}

\author[N.~Pronska]{Nataliya~Pronska}%

\address[N.P.]{Institute for Applied Problems of Mechanics and Mathematics,
3b~Naukova st., 79601 Lviv, Ukraine}
\email{nataliya.pronska@gmail.com}



\date{\today}
%\dedicatory{}%
%\commby{}%

\begin{abstract}
We study spectral properties of energy-dependent Sturm--Liouville
equations, introduce the notion of norming constants and establish
their interrelation with the spectra. One of the main tools is the
linearization of the problem in a suitable Pontryagin
space.
\end{abstract}

\maketitle

%%%%%%%%%%%%%%%%%%%%%%%%%%%%%%%%%%%%%%%%%%%%%%%%%%%%%%%%%

\section{Introduction}

%%%%%%%%%%%%%%%%%%%%%%%%%%%%%%%%%%%%%%%%%%%%%%%%%%%%%%%%%

The main aim of the present paper is to investigate spectral
properties of Sturm--Liouville problems with energy-dependent
potentials given by the differential equations
\begin{equation}\label{eq:intr.spr}
    -y''+qy+2\lambda p y=\lambda^2y
\end{equation}
on~(0,1) and some boundary conditions. Here~$p$ is a real-valued function
from~$L_{2}(0,1)$, $q$ is a real-valued distribution from the
Sobolev space~$W_2^{-1}(0,1)$, and~$\la\in\bC$ is a spectral
parameter. (A detailed definition will be given in the next
section).

The spectral equation~\eqref{eq:intr.spr} is of importance in
classical and quantum mechanics. For example, such problems arise
in solving the Klein--Gordon equations, which describe the motion
of massless particles such as photons (see \cite{Jon:93,Naj:83}).
Sturm--Liouville energy-dependent equations are also used for
modelling vibrations of mechanical systems in viscous media
(see~\cite{Yam:90}). Note that in such models the spectral
parameter~$\la$ is related to the energy of the system, and this
motivates the terminology ``energy-dependent'' used for the
spectral problem of the form~\eqref{eq:intr.spr}.

The equations under study were also considered on the line and
discussed in the context of the inverse scattering theory (see,
e.~g. \cite{JauJea761, MeePiv01,SatSzm95, AktMee91, Kam081,
MakGus86, Tsu81},  and~\cite{HryPro:2012} for a more extensive
reference list). Some of their spectral properties in this context
were established in~\cite{MeePiv02}. The spectral
problem~\eqref{eq:intr.spr} on an interval with~$p\in
W_2^1[0,\pi]$ and~$q\in L_2[0,\pi]$ and with general boundary
conditions was also studied by Gasymov and Nabiev
in~\cite{GusNab00,Nab00}. An interesting approach to the spectral
analysis of problems under consideration uses the theory of Krein
spaces (i.e. spaces with indefinite scalar products). It was
suggested by P.~Jonas~\cite{Jon:93} and H.~Langer, B.~Najman, and
C.~Tretter~\cite{LanNajTre:06,Naj:83,LanNajTre:08}.

 In the present paper, we consider~\eqref{eq:intr.spr} under
minimal smoothness assumption on the real-valued potentials~$p$
and~$q$. As equation~\eqref{eq:intr.spr} contains terms depending
on the spectral parameter~$\la$ and its square~$\la^2$ as well,
the spectral problem of interest is better understood as that for
the corresponding quadratic operator pencil. And indeed, some of
the spectral properties of the Sturm--Liouville energy-dependent
equations~\eqref{eq:intr.spr} are derived in this paper from the
general spectral theory of polynomial operator pencils
(see~\cite{Mar:88}) and some by the direct analysis of the
corresponding quadratic operator pencil. We also prove equivalence
of the spectral problem for~\eqref{eq:intr.spr} and that for its
linearization $\mathscr{L}$. The operator $\mathscr{L}$ turns out
to be self-adjoint in a suitably defined Pontryagin space, which
provides some further properties of the operator
pencil~\eqref{eq:intr.spr}.

We also introduce the notion of the norming constants for the
problem~\eqref{eq:intr.spr}. For real and simple eigenvalues the
definition of these quantities is analogous to that for the
standard Sturm--Liouville operators. However, since the
problem~\eqref{eq:intr.spr} can also have non-real and/or
non-simple eigenvalues our definition is more general. These
quantities are shown to be related to the spectra
of~\eqref{eq:intr.spr} similarly to the classical Sturm--Liouville
theory. We obtain an explicit formula determining norming
constants via two spectra. We also derive sufficient conditions
for simplicity of the spectra. The obtained results have their
important applications in the inverse problems of reconstruction
of the potentials~$p$ and~$q$ (or its primitive~$r$) from two
spectra or one spectrum and a set of norming constants
(see~\cite{HryPro:2012,Pro:2011c}).

The paper is organized as follows. In the next section we
formulate the spectral problem under study as that for the
corresponding operator pencil~$T$ and recall some notions from the
operator pencil theory. In Section~\ref{sec:SPOP}, we analyse the
operator pencil~$T$ and obtain some of its spectral properties. We
construct a linearization~$\cL$ of the spectral problem for~$T$ in
Section~\ref{sec:Lin}. The operator~$\cL$ is considered in a
specially defined Pontryagin space (i.~e. in the space with
indefinite inner product) and is shown to be self-adjoint therein.
This gives more spectral properties of~$\cL$ and so of~$T$. In
Section~\ref{sec:nc}, we introduce the notion of norming constants
for the problem under study and derive some relations for these
quantities. Section~\ref{sec:case} is devoted to the case when the
spectra of the problems~\eqref{eq:intr.spr} under two types of
boundary conditions are real and simple. We obtain sufficient
conditions for this.

\emph{Notations.} Throughout the paper, $\rho(T)$, $\sigma(T)$
and~$\sigma_\mathrm{p}(T)$ denote the resolvent set, the spectrum
and the point spectrum of a linear operator or a quadratic
operator pencil~$T$. The superscript~$\mathrm{t}$ will signify
 the transposition of vectors and matrices, e.~g.\
$(c_1,c_2)^{\mathrm{t}}$ is the column vector~$\binom{c_1}{c_2}$.

%%%%%%%%%%%%%%%%%%%%%%%%%%%%%%%%%%%%%%%%%%%%%%%%%%%%%%%%%
\section{Preliminaries}\label{sec:pre}
%%%%%%%%%%%%%%%%%%%%%%%%%%%%%%%%%%%%%%%%%%%%%%%%%%%%%%%%%%
Consider equation~\eqref{eq:intr.spr} subject to the Dirichlet
boundary conditions
\begin{equation}\label{eq:pre.b.c.}
    y(0)=y(1)=0.
\end{equation}
Notice that other separate boundary conditions can be treated
similarly; in particular, in Sections~\ref{sec:nc}
and~\ref{sec:case} we shall consider~\eqref{eq:intr.spr} under the
mixed conditions~\eqref{eq:nc.mbc}. We restrict our attention
to~\eqref{eq:pre.b.c.} merely in order to enlighten the ideas and
avoid unessential technicalities.

The spectral equation~\eqref{eq:intr.spr} depends on the
parameter~$\la$ non-linearly. Thus to formulate the spectral
problem of interest rigorously we should
regard~\eqref{eq:intr.spr} as a spectral problem for some operator
pencil. To start with, consider the differential expression
\[
    \ell(y) := - y'' + q y.
\]
As~$q$ is a real-valued distribution from~$W_2^{-1}(0,1)$ we need
to explain how~$\ell(y)$ is defined. The simplest and most
convenient way uses the method of regularization by
quasi-derivatives (see, e.g.~\cite{SavShk:1999,SavShk:2003}) that
proceeds as follows. Take a real-valued~$r\in L_2(0,1)$ such
that~$q=r'$ in the distributional sense and for every absolutely
continuous function~$y$ denote by~$y^{[1]}:=y'-ry$ its
quasi-derivative. We then define~$\ell$ as
\[
   \ell (y) = -\bigl(y^{[1]}\bigr)' - r y^{[1]} - r^2 y
\]
on the domain
\[
    \dom \ell = \{y \in AC (0,1) \mid y^{[1]} \in AC[0,1], \ \ell(y) \in L_2(0,1)\}.
\]
Direct verification shows that with this
definition~$\ell(y)=-y''+qy$ in the distributional sense. Observe
also that for every~$f$ from~$L_2(0,1)$, every complex~$a,b$ and
every~$x_0$ from~$[0,1]$ the equation~$\ell(y)=\mu y +f$ possesses
a unique solution satisfying the initial conditions~$y(x_0)=a$
and~$y^{[1]}(x_0)=b$.

Denote by~$A$ the operator acting via
\[
    Ay:=\ell(y)
\]
on the domain
\[
    \dom A:=\{y \in \dom \ell \mid y(0)=y(1)=0\}.
\]

For regular~$q$, the operator~$A$ is a standard Sturm--Liouville
operator with potential~$q$ and the Dirichlet boundary conditions.
It was shown in~\cite{SavShk:1999,SavShk:2003} that if~$q\in
W_2^{-1}(0,1)$ is real-valued, then the operator~$A$ is
self-adjoint, bounded below and has a simple discrete spectrum.

\begin{remark}\label{rem:pre.DisSp}
Recall that an operator~$S$ is said to possess discrete spectrum
if~$\sigma(S)$ consists of isolated points, each of which is an
eigenvalue of finite algebraic multiplicity. By Theorem~III.6.29
of~\cite{Kat:1966}, $S$ has discrete spectrum if its resolvent is
compact for one (and then for all)~$\lambda\in\rho(S)$.
\end{remark}

Next we denote by~$B$ the operator of multiplication by the
function ~$2p\in L_2(0,1)$, by~$I$ the identity operator and
define the quadratic operator pencil~$T$ as
\begin{equation}\label{eq:pre.T}
    T(\la):=\la^2I-\la B-A, \qquad \lambda \in \mathbb{C}.
\end{equation}
Then the spectral problem~\eqref{eq:intr.spr}, \eqref{eq:pre.b.c.}
can be regarded as the spectral problem for the operator
pencil~$T$. Properties of the operators~$A$ and~$B$ guarantee that
the pencil~$T$ is well defined on the~$\la$-independent
domain~$\dom T:=\dom A$. More exactly, the following statement
holds true.

\begin{proposition}\label{pr:pre.T_cl}
For every fixed~$\la_0 \in \mathbb{C}$ the operator~$T(\la_0)$ is
closed on the domain~$\dom T:=\dom A$ and has a discrete spectrum.
\end{proposition}

\begin{proof}
 Since the domain of the operator~$A$
consists only of bounded functions we have that~$\dom B\supset
\dom A$. This immediately gives that for every~$\la_0\in\bC$ the
operator~$T(\la_0)$ is well defined.

Let us fix~$\la_0\in\bC$.  Take an arbitrary~$\mu \in \rho(A)$ and
denote by~$\varphi_-$ and~$\varphi_+$ solutions of the
equation~$\ell (y)=\mu y$ satisfying boundary
conditions~$\varphi_-(0)=0$, $\varphi_-^{[1]}(0)=1$ and
~$\varphi_+(1)=0$, $\varphi_+^{[1]}(1)=1$. Then the Green function
of~$A-\mu$,~i.~e. the kernel of the operator~$(A-\mu I)^{-1}$ is
equal to
\[
k_0(x,s):=\left\{\begin{array}{c}
       \varphi_+(x)\varphi_-(y)/W, \text{ when } x>s \\
                     \varphi_-(x)\varphi_+(y)/W, \text{ when } x\le s\\
                  \end{array},\right.
\]
where~$W=\varphi_-(x)\varphi_+^{[1]}(x)-\varphi_+(x)\varphi_-^{[1]}(x)$
is the Wronskian of solutions~$\varphi_-$ and~$\varphi_+$. In
particular, the Green function is continuous on the
square~$\Omega:=[0,1]\times[0,1]$. It follows that the
operator~$(\la_0^2I-\la_0 B)(A-\mu I)^{-1}$ is an integral one
with the kernel~$k$ given by
\[
    k(x,s)=(\la_0^2-2\la_0 p(x))k_0(x,s).
\]
As~$k$ is square integrable on~$\Omega$, the
operator~$(\la_0^2I-\la_0 B)(A-\mu I)^{-1}$ is of the
Hilbert--Schmidt class and thus~$\la_0^2I-\la_0 B$ is~$A$--compact
(see~\cite[Ch.~IV]{Kat:1966}). In view of Theorem~IV.1.11
of~\cite{Kat:1966} the operator~$T(\la_0)$ is closed on~$\dom A$.
Moreover, Theorem~IV.5.35 of~\cite{Kat:1966} implies the
coincidence of the essential spectra of the operators $A$ and
$T(\lambda_0)$. As~$A$ has discrete spectrum, we get
that~$\sigma_{\mathrm{ess}}T(\lambda_0)=\sigma_{\mathrm{ess}}(A)=\emptyset$
and thus the spectrum of~$T(\lambda_0)$ is discrete.
\end{proof}

Let us now recall some notions of the spectral theory of operator
pencils, see~\cite{Mar:88}.

An \emph{operator pencil}~$T$ is an operator-valued function
on~$\bC$.
 The \emph{spectrum} of an
operator pencil~$T$ is the set~$\sigma(T)$  of all $\lambda\in\bC$
such that~$T(\lambda)$ is not boundedly invertible, i.e.
\[
    \sigma(T)=\{\lambda\in\mathbb{C}\mid 0\in\sigma(T(\lambda))\}.
\]
A number $\lambda\in\bC$ is called an \emph{eigenvalue} of $T$ if
$T(\lambda)y=0$ for some non-zero function~$y\in\dom T$, which is
then the corresponding \emph{eigenfunction}. The eigenvalues
of~$T$ constitute its \emph{point
spectrum}~$\sigma_\mathrm{p}(T)$, i.e.,
\[
    \sigma_\mathrm{p}(T)=\{\lambda\in \mathbb{C}\mid0\in\sigma_\mathrm{p}(T(\lambda))\}.
\]
 The set
\[
    \rho(T):=\mathbb{C}\setminus\sigma(T)
\]
is the \emph{resolvent set} of an operator pencil~$T$.

Vectors~$y_1,...y_{m-1}$ are said to be associated with an eigenvector~$y_0$ corresponding to an eigenvalue~$\la$ if
\[
    \sum\limits_{k=0}^j\frac{1}{k!}T^{(k)}(\la)y_{j-k}=0,\quad {j=1,...,m-1}.
\]
Here~$T^{(k)}$ denotes the \emph{k}-th derivative of~$T$ with
respect to~$\la$. The number~$m$ is called the length of the
chain~$y_0,\dots,y_{m-1}$ of an eigen- and associated vectors. The
maximal length of a chain starting with an eigenvector~$y_0$ is
called the \emph{algebraic multiplicity} of an eigenvector~$y_0$.

For an eigenvalue~$\la$ of~$T$ the dimension of the null-space
of~$T(\la)$ is called the \emph{geometric multiplicity} of~$\la$.
The eigenvalue is said to be \emph{geometrically simple} if its
geometric multiplicity equals to one.

For the pencil~$T$ of~\eqref{eq:pre.T} the operator~$-T(\la_0)$ is
a Sturm--Liouville operator with potential~$q+2\la_0p-\la_0^2$ and
the Dirichlet boundary conditions, whence the dimension of its
null-space is at most one. Therefore all the eigenvalues of the
pencil~$T$ under study are geometrically simple, and then the
\emph{algebraic multiplicity} of an eigenvalue is the algebraic
multiplicity of the corresponding eigenvector. (If the
eigenvalue~$\la$ is not geometrically simple, its algebraic
multiplicity is the number of vectors in the corresponding
canonical system, see~\cite{Mar:88,Kel:1971}). An eigenvalue is
said to be \emph{algebraically simple} if its algebraic
multiplicity is one.

%%%%%%%%%%%%%%%%%%%%%%%%%%%%%%%%%%%%%%%%%%%%%%%%%%%%%

\section{Spectral properties of the operator
pencil}\label{sec:SPOP}

%%%%%%%%%%%%%%%%%%%%%%%%%%%%%%%%%%%%%%%%%%%%%%%%%%%%%

In this section we discuss some basic spectral properties of the
operator pencil~$T$. We start with the following lemmas.
\begin{lemma}\label{lem:spop.sp_point}
The spectrum of the operator pencil~$T$ consists only of
eigenvalues.
\end{lemma}
\begin{proof}
By definition,~$\lambda_0\in\bC$ belongs to the spectrum of the
operator pencil~$T$ if and only if~$0\in\sigma(T(\lambda_0))$.
Since~$\sigma(T(\la_0))=\sigma_{\mathrm{p}}(T(\la_0))$ (see
Proposition~\ref{pr:pre.T_cl}), every~$\la_0$ in the spectrum
of~$T$ is its eigenvalue.
\end{proof}

\begin{lemma}\label{lem:spop.resol_set}
The resolvent set of the operator pencil~$T$ is not empty.
\end{lemma}
\begin{proof}
As the operator~$A$ is lower semibounded, a number~$\mu$ exists
such that the operator~$A+\mu^2 I$ is positive. Let us show that
then the number~$i\mu$ belongs to the resolvent set~$\rho(T)$ of
the operator pencil~$T$. Suppose it does not; then, by the
previous lemma, there exists a nonzero eigenfunction~$y$ such
that~$T(i\mu)y=0$ and so
\[
    ((\mu^2+A)y,y)+i\mu (B y,y)=0.
\]
This contradicts  positivity of~$A+\mu^2I$. Therefore~$i\mu$
belongs to~$\rho(T)$ and the lemma is proved.
\end{proof}

Using lemmas we shall prove discreteness of the spectrum of the
operator pencil~$T$.

\begin{lemma}\label{lem:psp.T_dis_sp}
The spectrum of the operator pencil~$T$ is a discrete subset
of~$\bC$.
\end{lemma}
\begin{proof}
Let us take some~$\la_0\in \rho(T)$ and rewrite~$T(\la)$ as
\[
    T(\la)=T(\lambda_0)+(\lambda-\lambda_0)[2\lambda_0I-B]+(\lambda-\lambda_0)^2I.
\]
Set~$\widehat{B}:=2\la_0I-B$,~$\widehat{A}:=T(\la_0)$,
and~$\mu=\la-\la_0$ and consider the operator
pencil~$\widehat{T}(\mu):=T(\la)T^{-1}(\la_0)$, which can be
written as
\[
    \widehat{T}(\mu):=I+\mu^2\widehat{A}^{-1}+\mu\widehat{B}\widehat{A}^{-1}.
\]
Using the arguments analogous to those used in the proof of
Proposition~\ref{pr:pre.T_cl} one can show that the
operator~$\mu^2\widehat{A}^{-1}+\mu\widehat{B}\widehat{A}^{-1}$ is
from the Hilbert--Schmidt class and so is compact. Then applying
the Gohberg theorem on analytic operator-valued
functions~\cite[Ch.I]{GohKre:1969} to the pencil~$I-S(\mu)$
with~$S(\mu):=-(\mu^2\widehat{A}^{-1}+\mu\widehat{B}\widehat{A}^{-1})$,
we obtain that for all~$\mu\in\bC$ except possibly some isolated
points the operator~$\widehat{T}(\mu)$ is boundedly invertible,
while these isolated points are eigenvalues of~$\widehat{T}$ of
finite algebraic multiplicity. This shows that the spectrum
of~$\widehat{T}$ is a discrete subset of~$\bC$.

Assume~$\la\in\sigma(T)$, which by Lemma~\ref{lem:spop.sp_point}
means that~$\la\in\sigma_{\mathrm{p}}(T)$, and let~$x$ be the
corresponding eigenfunction. Then~$y=T^{-1}(\la_0)x$ is an
eigenfunction of~$\widehat{T}$ corresponding to the
eigenvalue~$\mu=\la-\la_0$. Therefore,
\[
    \la\in\sigma(T)\Rightarrow
    \mu=\la-\la_0\in\sigma(\widehat{T}).
\]
Observe also that if~$ \la\in\rho(T)$,~i.~e. if~$T(\la)$ is
boundedly invertible, then the operator~$T(\la_0)T^{-1}(\la)$ is
closable, defined on the whole space $L_2(0,1)$, and thus bounded
by the closed graph theorem~\cite[Theorem~III.5.20]{Kat:1966}.
Direct verification shows that it is the inverse operator
of~$\widehat{T}(\mu)$ with~$\mu=\la-\la_0$. Therefore
\[
   \la\in\rho(T)\Rightarrow \mu=\la-\la_0\in\rho(\widehat{T}).
\]
These two implications give the equivalence
\[
    \la\in\sigma(T)\Leftrightarrow
    \mu=\la-\la_0\in\sigma(\widehat{T});
\]
thus the spectrum of the operator pencil~$T$ is discrete in
$\mathbb{C}$ along with the spectrum of~$\widehat{T}$.
\end{proof}

\begin{remark}\label{rem:0inRho}
Without loss of generality we may and shall assume further in this
paper that~$0$ is not in~$\sigma(T)$ or, equivalently, that the
operator~$A$ is boundedly invertible. In view of the above lemma
we can always achieve this by shifting of the spectral parameter
by a real number.
\end{remark}

As was noted in Section~\ref{sec:pre}, every eigenvalue of~$T$ is
geometrically simple. However, in general the spectrum of the
operator pencil~$T$ is not necessarily real or algebraically
simple as the following example demonstrates.

\begin{example}
Consider the operator pencil
\[
T(\lambda):=\lambda^2-2\lambda
\pi+\frac{d^2}{dx^2}+5\pi^2=(\la-\pi)^2+4\pi^2+\frac{d^2}{dx^2},
\]
i.e. the pencil $T$ with $p\equiv\pi$ and $q=r'\equiv-5\pi^2$.
Then~$\lambda_{\pm1}=(1\pm i\sqrt{3})\pi$ are complex conjugate
eigenvalues of this operator pencil, while~$\lambda_2=\pi$ is its
eigenvalue of algebraic multiplicity at least 2,
since~$y_0=\sin2\pi x$ and~$y_1\equiv0$ form the corresponding
chain of eigen- and associated vectors.
\end{example}

We summarize the above considerations in the following theorem.
\begin{theorem}
The spectrum of the operator pencil~$T$ of~\eqref{eq:pre.T} is a
discrete subset of~$\bC$ and consists of geometrically simple
eigenvalues.
\end{theorem}

%%%%%%%%%%%%%%%%%%%%%%%%%%%%%%%%%%%%%%%%%%%%%%%%%%%%%%%%%%%%%%%%%%%

\section{Linearization and its properties}\label{sec:Lin}

%%%%%%%%%%%%%%%%%%%%%%%%%%%%%%%%%%%%%%%%%%%%%%%%%%%%%%%%%%%%%%
In this section we shall recast the spectral problem for the
operator pencil~$T$ as a spectral problem for some linear
operator~$\cL$ and show equivalence of these problems. Considering
$\mathcal{L}$ in a specially defined Pontryagin space will then
reveal some further spectral properties of the pencil~$T$.

%%%%%%%%%%%%%%%%%%%%%%%%%%%%%%%%%%%%%%%%%%%%%%%%%%%%%%%%%%%%%%%%%%%
\subsection{Linearization}
%%%%%%%%%%%%%%%%%%%%%%%%%%%%%%%%%%%%%%%%%%%%%%%%%%%%%%%%%%%%%%%%%%%

Setting~$u_1:=y$ and~$u_2:=\lambda y$, we recast the
problem~\eqref{eq:intr.spr}--\eqref{eq:pre.b.c.} as the first
order system
\begin{equation}\label{eq:sys}
\begin{aligned}
u_2&=\la u_1 \\
Au_1+Bu_2&=\la u_2.
\end{aligned}
\end{equation}
The system~\eqref{eq:sys} is the spectral problem for the operator
\begin{equation}\label{eq:Lin.L0}
\cL_0:=\left(%
\begin{array}{cc}
  0 & I \\
  A & B \\
\end{array}%
\right).
\end{equation}
Therefore the spectral properties of the operator pencil~$T$
should be closely related to those of the operator~$\cL_0$. The
latter should be considered in the so called energy space $\cE$
which we next define. Recall that  the operator~$A$ is supposed to
be boundedly invertible (see Remark~\ref{rem:0inRho}). Denote
by~$H$ the space~$L_2(0,1)$ and by~$H_{\alpha}$, $\al\in\bR$, the
scale of Hilbert spaces generated by the operator~$A$. Thus the
space~$H_0$ coincides with~$H$, for any~$\al>0$ the space~$H_\al$
is the domain of the operator~$|A|^{\al}$ endowed with the
norm~$\|x\|_\al:=\||A|^\al x\|$, and for~$\al<0$ the space~$H_\al$
is the completion of~$H$ by the norm~$\|\cdot\|_\al$. Since the
operator~$A$ has compact resolvent for every~$\beta<\al$, the
embedding~$H_\al\hookrightarrow H_\beta$ is compact. Note that for
any~$\al>\theta$ the restriction of the operator~$A^{\al}$
to~$A^{\al}:H_\theta\rightarrow H_{\theta-\al}$ is homeomorphism.
Similarly, for ~$\al<\theta$ the extension of the
operator~$A^{\al}$ to~$A^{\al}:H_\theta\rightarrow H_{\theta-\al}$
is homeomorphism.

Introduce the Hilbert space~$(\cE, (\cdot,\cdot)_\cE)$,
where~$\cE:=H_{1/2}\times H$ and the scalar
product~$(\cdot,\cdot)_\cE$ is given by
\[
    (\mathbf{x},\mathbf{y})_\cE=(|A|^{1/2}x_1,|A|^{1/2}y_1)+(x_2,y_2)
\]
for every~$\mathbf{x}=(x_1,x_2)$ and $\mathbf{y}=(y_1,y_2)$
in~$\cE$. Then the operator~$\cL_0$ of~\eqref{eq:Lin.L0} is well
defined on the domain
\[
\dom \cL_0:=\{(u_1,u_2)^\mathrm{t}\mid u_1\in H_1;\; u_2\in
H_{1/2}\cap \dom B\}.
\]
However~$\cL_0$ is not closed on this domain. To describe its
closure, we need the following auxiliary result.
\begin{lemma}\label{lem:lin.B_comp}
The operator~$B$ extends by continuity to a compact
mapping~$\widetilde B$ from~$H_{1/2}$ to~$H_{-1/2}$.
\end{lemma}
\begin{proof}
Using the arguments analogous to those in the proof of
Proposition~\ref{pr:pre.T_cl} one can show that the
operator~$BA^{-1}:H\rightarrow H$ is compact. This yields the
compactness~$B:H_1\rightarrow H$ as a mapping from~$H_1$ to~$H$.

Observe that the space~$H_{-1}$ is dual to~$H_1$ with respect to
the scalar product~$(\cdot,\cdot)$. Denoting
by~$\langle\cdot,\cdot\rangle$ the pairing between~$H_{-1}$
and~$H_1$, we get for~$x\in H_1$ that
\begin{align*}
\|Bx\|_{-1} & = \sup_{y\in H_1\,:\,\|y\|_1=1}|\langle Bx,y
\rangle|  = \sup_{y\in H_1\,:\,\|y\|_1=1}|(Bx,y)| \\ &= \sup_{y\in
H_1\,:\,\|y\|_1=1}|(x,By)| \le \|x\|_0\|B\|_{H_1\to H_0}.
\end{align*}
Therefore~$B$ extends by continuity to a bounded mapping
from~$H_0$ to~$H_{-1}$. Now using the interpolation theorem for
compact operators~\cite{Per:63} we obtain compactness
of~$\widetilde{B}:H_{1/2}\rightarrow H_{-1/2}$.
\end{proof}

Next observe that the operator~$A$ can be extended by continuity
to a homeomorphism~$\widetilde{A}:H_{1/2}\rightarrow H_{-1/2}$.

\begin{lemma}[\hspace*{-5pt}\cite{HrShk:2004}]
The operator~$\cL_0$ is closable and the closure~$\cL$ is given by
the formulae
\begin{equation}
\begin{aligned}
    &\cL\left(%
\begin{array}{c}
  x_1 \\
  x_2 \\
\end{array}%
\right)=\left(%
\begin{array}{c}
  x_2 \\
  \widetilde{A}x_1+\widetilde{B}x_2 \\
\end{array}%
\right),\\
\dom\cL&=\left\{(x_1,x_2)^\mathrm{t}\mid x_1,x_2\in H_{1/2},\;
\widetilde{A}x_1+\widetilde{B}x_2\in H\right\}.
\end{aligned}
\end{equation}
\end{lemma}

We are going to show the coincidence of the spectra of~$T$
and~$\cL$. To start with we prove that the spectrum of~$\cL$ is
discrete.

\begin{lemma}\label{lem:lin.L_sp_dis}
The spectrum of the operator~$\cL$ is discrete.
\end{lemma}
\begin{proof}
By Remark~\ref{rem:pre.DisSp}, to show discreteness of the
spectrum of~$\cL$ it is enough to establish that its
inverse~$\cL^{-1}$ is compact.

Consider the system
\begin{align*}
x_2&=y_1,\\
\widetilde{A}x_1+\widetilde{B}x_2&=y_2.
\end{align*}
with~$\mathbf{x}=(x_1,x_2)^\mathrm{t}$ from~$\dom\cL$
and~$\mathbf{y}=(y_1,y_2)^\mathrm{t}$ from~$\cE$. Since the
operator~$\widetilde{A}$ is boundedly invertible (see
Remark~\ref{rem:0inRho}) this system gives that the operator~$\cL$
is boundedly invertible and its inverse~$\cL^{-1}$ is given by the
matrix
\[
    \cL^{-1}=\left(
              \begin{array}{cc}
                -\widetilde{A}^{-1}\widetilde{B} & A^{-1} \\
                I & 0 \\
              \end{array}
            \right)
\]
on~$\cE$. Next we show that~$\cL^{-1}:\mathcal{E} \rightarrow
\mathcal{E}$ is a compact operator. To do this we show compactness
of all entries of the corresponding matrix.

The operator~$I:H_{1/2}\rightarrow H$ is an embedding of the
space~$H_{1/2}$ into~$H$ and thus it is compact.

Since the operator~$A^{-1}:H\rightarrow H_{1}$ is bounded and the
embedding~$H_1\hookrightarrow H_{1/2}$ is compact the
operator~$A^{-1}:H\rightarrow H_{1/2}$ is compact as the
composition of a bounded operator and a compact operator.

It was established in Lemma~\ref{lem:lin.B_comp} that the
operator~$\widetilde{B}:H_{1/2}\rightarrow H_{-1/2}$ is compact.
Since the operator~$\widetilde{A}^{-1}:H_{-1/2}\rightarrow
H_{1/2}$ is bounded this gives compactness
of~$\widetilde{A}^{-1}\widetilde{B}:H_{1/2}\rightarrow H_{1/2}$.

These observations yield compactness of~$\cL^{-1}$ and complete
the proof.
\end{proof}

For~$\la\in\bC$ we set
\begin{equation}
    \widetilde{T}(\la):=\la^2I-\la \widetilde{B}-\widetilde{A}.
\end{equation}
and consider~$\widetilde{T}(\la)$ as an operator from~$H_{1/2}$
to~$H_{-1/2}$.

\begin{theorem}[\hspace*{-5pt}\cite{HrShk:2004}]\label{thm:lin.L_Tt}
The spectrum of the operator~$\cL$ coincides with the
spectrum~$\sigma(\widetilde{T})$ of the operator
pencil~$\widetilde{T}$. For every
nonzero~$\la\in\rho(\widetilde{T})$ the following representation
holds:
\begin{equation}\label{eq:lin.repRes}
    (\cL-\la I)^{-1}=\left(%
\begin{array}{cc}
  -\la^{-1}(\widetilde{T}^{-1}(\la)\widetilde{A}+ I) & -\widetilde{T}(\la)^{-1} \\
  -\widetilde{T}^{-1}(\la)\widetilde{A} & -\la \widetilde{T}(\la)^{-1} \\
\end{array}%
\right).
\end{equation}
\end{theorem}

Now we can show coincidence of the spectra of the operator~$\cL$
and of the operator pencil~$T$. In view of
Lemmas~\ref{lem:psp.T_dis_sp} and~\ref{lem:lin.L_sp_dis}, it is
sufficient to show coincidence of the corresponding eigenvalues.

\begin{theorem}\label{thm:lin.coin}
The eigenvalues of the operator pencil~$T$ coincide with those of
the operator $\cL$ counting multiplicities.
\end{theorem}
\begin{proof}
Observe firstly that for every~$\lambda_0\in\mathbb{C}$ the
operator~$\widetilde T(\lambda_0)$ is an extension
of~$T(\lambda_0)$. Therefore if for
some~$\mu\in\rho(T(\la_0))\cap\rho(\widetilde{T}(\la_0))$ one
has~$(\widetilde{T}(\la_0)-\mu)u\in H$, then~$u$ belongs to~$H_1$,
i.e. to~$\dom T$. We shall use this remark in our further
discussions.

Assume that~$\la_0\in\bC$ is an eigenvalue of~$T$ with the
corresponding chain of eigen- and associated
vectors~$y_0,y_1,\dots, y_{m-1}$. By definition this means that
\[
    (\lambda_0^2-\lambda_0 B-A)y_k+(2\lambda_0-B)y_{k-1}+y_{k-2}=0
\]
for~$k=0,\dots, m-1$ with~$y_{-1}$, $y_{-2}$ being zero. A direct
verification shows that these equalities are equivalent to the
following
\[
    (\cL-\la_0)Y_k=Y_{k-1}, \quad k=0,\dots,m-1
\]
with~$Y_k=(y_k,\la y_k+y_{k-1})^{\mathrm{t}}$. In particular,
$\lambda_0$ is an eigenvalue of $\mathscr{L}$ and the
vectors~$Y_0,Y_1,\dots, Y_{m-1}$ belong to the domain of~$\cL$ and
so form a chain of eigen- and associated vectors of~$\cL$
corresponding to~$\la_0$.

Next suppose that~$\la_0\in\bC$ is an eigenvalue of~$\cL$ with the
corresponding chain of eigen- and associated
vectors~$Y_0,\dots,Y_{m-1}$ of length~$m$. By
definition,~$(\cL-\la_0)Y_k=Y_{k-1}$ ($Y_{-1}$ is supposed to be
zero) or, setting~$Y_k:=(Y_{k,1},Y_{k,2})^{\mathrm{t}}$,
\begin{align*}
-\la_0Y_{k,1}+Y_{k,2}&=Y_{k-1,1},\\
\widetilde{A}Y_{k,1}+(\widetilde{B}-\la_0)Y_{k,2}&=Y_{k-1,2}.
\end{align*}
This gives that~$Y_{k,2}=\la_0Y_{k,1}+Y_{k-1,1}$ and
\[
    (\la_0^2-\la_0 \widetilde{B}-\widetilde{A})Y_{k,1}+(2\la_0-\widetilde{B})Y_{k-1,1}+Y_{k-2,1}=0
\]
for~$k=0,\dots m-1$ and~$Y_{-1}=Y_{-2}=0$.
Since~$Y_k=(Y_{k,1},Y_{k,2})^{\mathrm{t}}\in \dom \cL$, we have
that~$Y_{k,1}$ is from~$H_{1/2}$. Thus the last equality yields
that~$Y_{0,1}, Y_{1,1}, \dots, Y_{m-1,1}$ is a chain of eigen- and
associated vectors of the operator pencil~$\widetilde{T}$
corresponding to the eigenvalue~$\la_0$.

Next we prove by induction that all the vectors~$Y_{0,1}, Y_{1,1},
\dots, Y_{m-1,1}$ belong to~$H_1$ and, therefore, form a chain of
eigen- and associated vectors of~$T$ corresponding to~$\la_0$.
Consider firstly the vector~$Y_0$.
Take~$\mu\in\rho(T(\la_0))\cap\rho(\widetilde{T}(\la_0))$ and
observe that~$(\widetilde{T}(\la_0)-\mu)Y_{0,1}=-\mu Y_{0,1}\in
H$. In view of remark made at the beginning of the proof, this
gives that~$Y_{0,1}$ belongs to~$H_1$ and so it is an eigenvector
of~$T$ corresponding to~$\la_0$. Now suppose that~$Y_{j,1}$
belongs to~$H_1$ for every~$j<k$.
Taking~$\mu\in\rho(\widetilde{T}(\la_0))\cap\rho(T(\la_0))$ we
obtain
\[
    (\widetilde{T}(\la_0)-\mu)Y_{k,1}=-\mu
    Y_{k,1}-(2\la_0-\widetilde{B})Y_{k-1,1}+Y_{k-2,1}\in H.
\]
By the same reason we deduce that~$Y_{k,1}$ is from~$H_1$; thus
the chain of eigen- and associated vectors of~$\cL$ generates the
corresponding chain for~$T$.

The above reasonings show that there is a one-to-one
correspondence between the chains of eigen- and associated vectors
of~$T$ and~$\cL$ corresponding to the same eigenvalues, thus
establishing the claim.
\end{proof}

Coincidence of the eigenvalues of~$\cL$ and~$T$ can be proved in
another way. Observe that Lemma~\ref{lem:lin.L_sp_dis} and
Theorem~\ref{thm:lin.L_Tt} imply that the spectrum
of~$\widetilde{T}$ is discrete. But it is known
(see~\cite{Shk:1996}) that the discrete parts of the spectra
of~$T$ and~$\widetilde{T}$ coincide. In view of
Proposition~\ref{lem:psp.T_dis_sp} this gives
that~$\sigma(T)=\sigma(\widetilde{T})$ and
so~$\sigma(T)=\sigma(\cL)$.

%%%%%%%%%%%%%%%%%%%%%%%%%%%%%%%%%%%%%%%%%%%%%%%%%%%%%%%%%%%%%%%%%%
\subsection{The Pontryagin space properties of~$\cL$}
%%%%%%%%%%%%%%%%%%%%%%%%%%%%%%%%%%%%%%%%%%%%%%%%%%%%%%%%%%%%%%%%%%

Now show that the linearization~$\cL$ is self-adjoint in some
Pontryagin space. Consider the operator~$J=P_+-P_-$, where~$P_{+}$
and~$P_-$ are the~$(\cdot,\cdot)$-orthogonal projectors onto the
spectral subspaces of $A$ corresponding to the positive and
negative parts
of the spectrum respectively. Set~$\mathscr{J}:=\left(%
\begin{array}{cc}
  J & 0 \\
  0 & I \\
\end{array}%
\right)$ and define an inner product
\[
    [\mathbf{x},\mathbf{y}]:=(\mathscr{J}x,y)_\cE=(J|A|^{1/2}x_1,|A|^{1/2}y_1)+(x_2,y_2)
\]
for every~$\mathbf{x}=(x_1,x_2)^\mathrm{t}$
and~$\mathbf{y}=(y_1,y_2)^{\mathrm{t}}$ from $\mathcal{E}$. The
number of negative eigenvalues of the operator~$\mathscr{J}$
equals that of~$J$, which in turn is the number of negative
eigenvalues of~$A$. But the operator~$A$ is lower semibounded and
the number of its negative eigenvalues is finite, say~$\kappa$.
This gives that the product~$[\cdot,\cdot]$ is indefinite and the
space~$\Pi:=(\mathcal{E},[\cdot,\cdot])$ is a Pontryagin space of
negative index~$\kappa$. Note that the topology of~$\Pi$ coincide
with that of~$\cE$.

Consider the operator~$\cL_0$ in the space~$\Pi$. For
every~$x=(x_1,x_2)^\mathrm{t}$ from $\dom \cL_0$ we have
\begin{align*}
    [\cL_0x,x]&=(J|A|^{1/2}x_2,|A|^{1/2}x_1)+(Ax_1,x_2)+(Bx_2,x_2)\\
    &=(Ax_1,x_2)+(x_2,Ax_1)+(x_2,Bx_2)=2 \mathrm{Re}(Ax_1,x_2)+(x_2,Bx_2).
\end{align*}
Clearly, this shows that~$[\cL_0x,x]$ is real and thus the
operator~$\cL_0$ is symmetric in~$\Pi$. This together with the
fact that the spectrum of~$\cL$ is discrete (see
Lemma~\ref{lem:lin.L_sp_dis}) gives the following proposition
(see~\cite{AziIoh:1989}).
\begin{proposition}
The operator~$\cL$ is self-adjoint in the Pontryagin space~$\Pi$.
\end{proposition}

Using this result and properties of self-adjoint operators in
Pontryagin spaces (see Proposition~\ref{pr:A.spSAoPs}) we obtain
that the spectrum of the operator~$\cL$ is real with possible
exception of at most~$\kappa$ pairs of complex-conjugate
eigenvalues~$\lambda$ and~$\bar{\lambda}$. The algebraic
multiplicity of any eigenvalue of~$\cL$ can not
exceed~$2\kappa+1$.

When the operator~$A$ is positive the space~$\Pi$ is a Hilbert
space (the negative index~$\kappa=0$). Then~$\cL$ is a
self-adjoint operator in a Hilbert space and thus its spectrum is
real. Moreover, the algebraic and geometric multiplicities of the
eigenvalues of~$\cL$ are equal. In view of geometric simplicity of
the eigenvalues of~$\cL$, this yields that the spectrum of~$\cL$
is algebraically simple. Therefore, if the operator~$A$ is
positive, then the spectrum of the operator~$\cL$ is real and
simple.

Summing up all these results and using the equivalence of spectral
problems for the operator pencil~$T$ and for the operator~$\cL$ we
obtain more spectral properties of~$T$ given in the following
proposition.

\begin{theorem}\label{thm:PS.SpT}
Let~$\kappa$ be the number of negative eigenvalues of the
operator~$A$. Then
\begin{itemize}
\item[(i)] the spectrum of the operator pencil~$T$ is real with possible
    exception of at most~$\kappa$ pairs of complex-conjugate
    eigenvalues~$\lambda$ and~$\bar{\lambda}$;
\item[(ii)] algebraic multiplicity of every eigenvalue of~$T$ does
    not exceed~$2\kappa+1$.
\end{itemize}
If the operator~$A$ is positive, then the spectrum of the operator
pencil~$T$ is real and simple.
\end{theorem}

%%%%%%%%%%%%%%%%%%%%%%%%%%%%%%%%%%%%%%%%%%%%%%%%%%%%%%%%%%%%%%%%%%%%%%%%%%%%%%%%%%%

\section{Norming constants}\label{sec:nc}

%%%%%%%%%%%%%%%%%%%%%%%%%%%%%%%%%%%%%%%%%%%%%%%%%%%%%%%%%%%%%%%%%%%%%%%%%%%%%%%%%%%
\subsection{Notion of norming constants} In this section we introduce
the notion of norming constants for the operator pencil~$T$ and
establish some of their properties. In the case where the
pencil~$T$ has only real and simple eigenvalues the norming
constants were used in~\cite{HryPro:2012} to solve the inverse
spectral problem of determination of the potentials~$p$ and~$q$ of
the pencil.

We say that a matrix is upper (lower) anti-triangular if all its
elements under (above) anti-diagonal are zero. Denote
by~$M^{+}[\gamma_1,\gamma_2,\dots\gamma_p]$, respectively
by~$M^{-}[\gamma_1,\gamma_2,\dots\gamma_p]$,  a Hankel upper,
respectively lower, anti-triangular matrices given by
\[
M^{+}[\gamma_1,\gamma_2,\dots\gamma_p]=\left(%
\begin{array}{cccc}
  \gamma_1 & \gamma_{2} & \cdots & \gamma_{p} \\
  \gamma_{2} & & \iddots  & 0 \\
  \vdots &  \iddots & \iddots & \vdots \\
  \gamma_{p} & 0 & \cdots & 0 \\
\end{array}%
\right)
\]
and
\[
    M^{-}[\gamma_1,\gamma_2,\dots,\gamma_p]=\left(%
\begin{array}{cccc}
  0 & \cdots & 0 & \gamma_{1} \\
  \vdots &\iddots & \iddots  & \gamma_{2} \\
  0 &  \iddots & \iddots & \vdots \\
  \gamma_{1} & \gamma_{2} & \cdots & \gamma_{p} \\
\end{array}%
\right).
\]
If some matrix~$M=M^{\pm}[\gamma_1,\dots,\gamma_p]$ we say that
the sequence~$\gamma_1,\dots,\gamma_p$ is associated with the
matrix~$M$.

 In this section we shall often work with infinite
block-diagonal matrices with upper (lower) anti-triangular blocks
of two types. The first type blocks are just upper (lower)
anti-triangular Hankel matrices. The second type ones have the
form
\begin{equation}\label{eq:Lin.bl2}
\left(%
\begin{array}{cc}
  0 & B_1 \\
  B_2 & 0\\
\end{array}%
\right),
\end{equation}
where~$B_1$ is upper (lower) anti-triangular Hankel matrix
and~$B_2$ is its complex conjugate. Denote the blocks of such a
matrix~$M$ by~$M_n$, $n\in\bZ$. To every bock~$M_n$ of size $m$
there is associated a number sequence of length $m$; these finite
sequences together form an infinite
sequence~$(\gamma_k)_{k\in\mathbb{Z}}$ associated with~$M$.

List the eigenvalues~$\la_k$, ${k\in\bZ}$, of the operator
pencil~$T$ so that
\begin{itemize}
    \item[(i)] each eigenvalue is repeated according to its multiplicity;
    \item[(ii)] the real parts of eigenvalues do not decrease, i.e.~$\mathrm{Re} \la_i\le \mathrm{Re} \la_j$ for~$i<j$;
    \item[(iii)]the moduli of the imaginary parts of the eigenvalues
    with equal real parts non-decrease, i.e.
    if~$\mathrm{Re}\la_i=\mathrm{Re}\la_j$ for some~$i<j$
    then~$|\mathrm{Im}\la_i|\le|\mathrm{Im}\la_j|$.
   Moreover, if~$|\mathrm{Im}\la_i|=|\mathrm{Im}\la_j|$ for
   some~$i\le j$ then~$\mathrm{Im}\la_i\ge\mathrm{Im}\la_j$.
\end{itemize}
This enumeration is such that if some~$\la$ is an eigenvalue
of~$T$ of multiplicity~$m$, then there is~$n\in\bZ$ such
that~$\la=\la_n=\la_{n+1}=\dots=\la_{n+m-1}$. If, moreover,~$\la$
is non-real,
then~$\overline{\la}=\la_{n+m}=\la_{n+m+1}=\dots=\la_{n+2m-1}$.
Along with the eigenvalue sequence~$(\la_k)_{k\in\bZ}$ we
introduce the sequence~$(y_k)_{k\in\bZ}$ of vectors from~$\dom A$
in the following way: if~$\la=\la_n=\la_{n+1}=\dots=\la_{n+m-1}$
is an eigenvalue of~$T$ of multiplicity~$m$,
then~$y_n,y_{n+1},\dots y_{n+m-1}$ is a chain of eigen- and
associated vectors of~$T$ corresponding to~$\la$ and such
that~$y_n(x,\la)$ satisfies the initial
conditions~$y_n(0,\la)=0$,~$y_n^{[1]}(0,\la)=\la$ and the
functions~$y_k,y_{k+1},...,y_{k+m-1}$ are defined via
\begin{equation}\label{eq:nc.evec}
y_{n+j}(x,\la):=\left.\frac{1}{j!}\frac{\partial^j
y_n(x,z)}{\partial z^j}\right|_{z=\la},\quad j=0,1,...,m-1.
\end{equation}
Note that for complex conjugate
eigenvalues~$\la=\la_n=\la_{n+1}=\dots=\la_{n+m-1}=\overline{\la_{n+m}}=\dots=\overline{\la_{n+2m-1}}$,
the vectors~$y_{j+m}=\overline{y_j}$ for~$j=n,\dots,n+m-1$.

Next define the \emph{norming constants}~$\al_n$, $n\in\bZ$, for
the operator pencil~$T$ as follows. For an
eigenvalue~$\la=\la_n=\la_{n+1}=\dots=\la_{n+m-1}$ we put
\begin{equation}\label{eq:nc.nc_def}
\begin{aligned}
\al_n&=\left(\frac{T'(\la)}{\la}y_n,y_n\right),\\
\al_{n+j}&=\left(\frac{T'(\la)}{\la}y_n,y_{n+j}\right)+\frac{1}{\la}(y_n,y_{n+j-1}),\quad
\text{for } j=1...m-1.
\end{aligned}
\end{equation}
 Defining the norming constants for the
operator pencil~$T$ in the described way is quite reasonable.
Firstly, note that for real and simple eigenvalues so defined
norming constants determine the type of eigenvalues
(see~\cite{Mar:88}) as
\[
    \bigl(T'(\la_n)y_n,y_n\bigr) =  \la_n\al_n,
\]
where~$T'$ is the~$\la$-derivative of~$T$. Secondly, if the
potential~$p$ is identically zero,~\eqref{eq:intr.spr} is the
spectral equation for Sturm--Liouville operator~$A$ and the given
definition of the norming constants for the operator pencil~$T$
coincides with the standard definition of norming
constants,~\cite{GelLev:51}. Further we shall see that so defined
norming constants are closely related to the norming constants for
the linearization~$\cL$.

Recall (see Theorem~\ref{thm:lin.coin}) that the
sequence~$(\la_k)_{k\in\bZ}$ is also an eigenvalues sequence for
the operator~$\cL$. Consider the sequence of
vectors~$(Y_k)_{k\in\bZ}$, such that for the
eigenvalue~$\la=\la_n=\la_{n+1}=\dots=\la_{n+m-1}$ of
multiplicity~$m$ the vectors~$Y_k$, $k=n,\dots, n+m-1$ are defined
by formulae $Y_n=(y_n,\la y_n)^{\mathrm t}$ and~$Y_j=(y_j,\la
y_j+y_{j-1})^{\mathrm{t}}$, $j=n+1,...n+m-1$. From the proof of
Theorem~\ref{thm:lin.coin} we know that~$Y_n,Y_{n+1},\dots
Y_{n+m-1}$ is the chain of eigen- and associated vectors of~$\cL$
corresponding to the eigenvalue~$\la$ and so~$(Y_n)_{n\in\bZ}$ is
the sequence of all eigen- and associated vectors of~$\cL$.

Let us put~$g_{kl}:=[Y_k,Y_l]$ and associate with the
operator~$\cL$ the Gramm matrix~$G=(g_{kl})$. In view of the
equalities
\[
    [Y_{k-1},Y_{l}]=[(\cL-\la
    I)Y_k,Y_l]=[Y_k,(\cL-\overline{\la}I)Y_l],
\]
for every real~$\la=\la_n=\dots=\la_{n+m-1}$ we have
$g_{ij}=g_{kl}$ for~$n\le i,j,k,l\le n+m-1$ such that~$i+j=k+l$.
Moreover,~$g_{kl}=0$ for~$k+l<2n+m-1$. Similarly, for
complex~$\la=\la_n=\dots=\la_{n+m-1}=\overline{\la_{n+m}}=\dots=\overline{\la_{n+2m-1}}$
we obtain~$g_{ij}=g_{kl}$ for~$n\le i,j,k,l\le n+2m-1$ such
that~$2n+m\le i+j=k+l\le 2n+3m-2$ and~$g_{kl}=0$ for~$2n+m\le
k+l<2n+2m-1$. Therefore the matrix $G$ is of block-diagonal form
with blocks of two types: blocks corresponding to real eigenvalues
and those corresponding to pairs of complex conjugate ones. The
block corresponding to a real
eigenvalue~$\la_n=\la_{n+1}=\dots=\la_{n+m-1}$ of multiplicity~$m$
is a Hankel lower anti-triangular
matrix~$M^{-}[\beta_n,\dots\beta_{n+m-1}]$. The sub-matrix of~$G$
corresponding to the pair of complex conjugate
eigenvalues~$\la_n=\la_{n+1}=\dots=\la_{n+m-1}=\overline{\la_{n+m}}=\overline{\la_{n+m+1}}=\dots=\overline{\la_{n+2m-1}}$
is of the form~\eqref{eq:Lin.bl2}
with~$B_1=M^{-}[\beta_n,...\beta_{n+m-1}]$ and its complex
conjugate~$B_2$ (see Proposition~\ref{pr:A.gram}). We call the
number~$\beta_k$, $k\in\bZ$, the \emph{norming constant} of~$\cL$
corresponding to the eigenvalue~$\la_k$.

Straightforward verification shows that the norming constants
of~$\cL$ are related with those of~$T$ as follows
\begin{equation}\label{eq:nc.ba}
\beta_k=\la_k^2\alpha_k
\end{equation}
for every~$k\in\bZ$.

\begin{theorem}\label{thm:nc.pos_al_A}
All the norming constants of~$T$ are positive if and only if the
operator~$A$ is positive.
\end{theorem}
\begin{proof}
Observe firstly that by~\eqref{eq:nc.ba} for nonzero eigenvalues
the norming constants of~$T$ are positive if and only if those
of~$\cL$ are also positive.

\emph{Sufficiency.} Obviously, if~$A$ is positive then the
space~$\Pi$ is a Hilbert space and so~$\cL$ is a self-adjoint
operator in a Hilbert space and all its norming constants are
positive as norms of eigenvectors in Hilbert space.

\emph{Necessity.} Suppose that~$A$ is not positive. Then~$\Pi$ is
a Pontryagin space of finite negativity index, say~$\kappa$, and
so by the Pontryagin theorem it has a maximal non-positive
subspace of dimension~$\kappa$ invariant with respect to~$\cL$.
Therefore~$\cL$ possesses an eigenvector in this subspace, and the
norming constant generated by this eigenvector is non-positive.
Thus not all norming constants of~$T$ are positive, and the proof
is complete.
\end{proof}

%%%%%%%%%%%%%%%%%%%%%%%%%%%%%%%%%%%%%%%%%%%%%%%%%%%%%%%%%%%%%%%%%%%
\subsection{Relations for norming constants}
%%%%%%%%%%%%%%%%%%%%%%%%%%%%%%%%%%%%%%%%%%%%%%%%%%%%%%%%%%%%%%%%%%%%%

Let us now consider the vector~$X=(0,x)^{\mathrm{t}}$ and compute
the residues of~$(\cL-z)^{-1}X$ at an eigenvalue~$\la$ in two
different ways. Equating the results, we shall obtain some
relations for norming constants~$\al_k$.

Denote by~$D=(d_{kl})$ the matrix inverse to~$G$,~i.e. $D=G^{-1}$.
Observe that~$D$ as well as~$G$ has a block-diagonal structure but
with upper anti-triangular matrices in blocks. Associate with~$D$
the sequence~$(\delta_k)_{k\in\bZ}$.

 Since the system of eigen- and associated vectors~$Y_j$ of~$\cL$ forms a Riesz basis in the Pontryagin space~$\Pi$
(see Proposition~\ref{pr:A.spSAoPs}) we can write the resolution
of identity in~$\Pi$ as follows
\[
\mathscr{I}=\sum\limits_{k=-\infty}^{+\infty}\sum\limits_{l=-\infty}^{+\infty}d_{kl}[\cdot,Y_l]Y_k,
\]
where the sum is convergent in the strong operator topology.
Using this it is straightforward to compute the residue
of~$(\cL-z)^{-1}X$ at the real
eigenvalue~$\la=\la_n=\la_{n+1}=\dots=\la_{n+m-1}$ as
\begin{equation}\label{eq:nc.res1r}
    \myres\limits_{z=\la}(\cL-z)^{-1}X=-\sum\limits_{k=n}^{n+m-1}\sum\limits_{l=n}^{n+m-1}d_{kl}[X,Y_l]Y_k=-\sum\limits_{k=n}^{n+m-1}\sum\limits_{l=n}^{n+m-1}d_{kl}(x,Y_{l,2})Y_k
\end{equation}
and for
complex~$\la=\la_n=\dots=\la_{n+m-1}=\overline{\la_{n+m}}=\dots=\overline{\la_{n+2m-1}}$
\begin{equation}\label{eq:nc.res1c}
    \myres\limits_{z=\la}(\cL-z)^{-1}X
    =-\sum\limits_{k=n}^{n+m-1}\sum\limits_{l=n+m}^{n+2m-1}d_{kl}[X,Y_l]Y_k
    =-\sum\limits_{k=n}^{n+m-1}\sum\limits_{l=n+m}^{n+2m-1}d_{kl}(x,Y_{l,2})Y_k.
\end{equation}

Next using the representation~\eqref{eq:lin.repRes} we obtain
\[
    (\cL-zI)^{-1}X=\left(%
\begin{array}{c}
  -T(z)^{-1}x \\
  -z T(z)^{-1}x \\
\end{array}%
\right).
\]

By Green's Formula
\[
    T(z)^{-1}f(x)=\frac{1}{W(z)}\left[\varphi(x,z)\int_x^1\!f(t)\psi(t,z)dt+\psi(x,z)\int_0^x\!f(t)\varphi(t,z)dt\right],
\]
where~$\varphi(\cdot,z)$ is a solution of the
equation~$\ell(y)=(z^2-2z p)y$ satisfying the initial
conditions~$y(0)=0$, $y^{[1]}(0)=z$,~$\psi(\cdot,z)$ is a solution
of the same equation satisfying the conditions~$y(1)=0$,
$y^{[1]}(1)=z$ and~$W(z)$ is the Wronskian of the
solutions~$\varphi(\cdot,z)$
and~$\psi(\cdot,z)$,~i.e.~$W(z)=\varphi(x,z)\psi^{[1]}(x,z)-\psi(x,z)\varphi^{[1]}(x,z)$.
Set~$s(z):=\varphi(1,z)$ and~$c(z):=\varphi^{[1]}(1,z)$. Since the
Wronskian~$W$ does not depend on~$x$ we observe that~$W(z)=z
\varphi(1,z)=z s(z)$. Next, note that for an eigenvalue~$\la$
of~\eqref{eq:intr.spr}, \eqref{eq:pre.b.c.} the
functions~$\psi(x,\la)$ and~$\varphi(x,\la)$ are related as
follows
\[
    \psi(x,\la)=\frac{\psi^{[1]}(1,\la)}{\varphi^{[1]}(1,\la)}\varphi(x,\la)=\frac{\la}{c(\la)}\varphi(x,\la).
\]
Taking these remarks into account, we compute the residues
of~$T(z)^{-1}$ and~$z T(z)^{-1}$ at the real
eigenvalue~$\la=\la_n=...=\la_{n+m-1}$ of multiplicity~$m$
\begin{align*}
    \myres\limits_{z=\la}T(z)^{-1}x&=\sum\limits_{k=n}^{n+m-1}\sum\limits_{l=n}^{n+m-1}h_{kl}(x,y_l)y_k,\\
     \myres\limits_{z=\la}z
     T(z)^{-1}x&=\sum\limits_{k=n}^{n+m-1}\sum\limits_{l=n}^{n+m-1}h_{kl}(x,y_l)(\la y_k+y_{k-1}),
\end{align*}
and at the complex
eigenvalue~$\la=\la_n=\la_{n+1}=...=\la_{n+m-1}=\overline{\la_{n+m}}=\overline{\la_{n+m+1}}=...=\overline{\la_{n+2m-1}}$
\begin{align*}
    \myres\limits_{z=\la}T(z)^{-1}x&=\sum\limits_{k=n}^{n+m-1}\sum\limits_{l=n+m}^{n+2m-1}h_{kl}(x,y_l)y_k,\\
     \myres\limits_{z=\la}z
     T(z)^{-1}x&=\sum\limits_{k=n}^{n+m-1}\sum\limits_{l=n+m}^{n+2m-1}h_{kl}(x,y_l)(\la
     y_k+y_{k-1})
\end{align*}
where~$h_{kl}$ form a matrix~$H=(h_{kl})$ of block diagonal form
with upper anti-triangular blocks of two types. The elements of
the sequence~$(\eta_k)_{k\in\bZ}$ associated with~$H$ are given
explicitly in the following way. For real
eigenvalue~$\la=\la_n=...=\la_{n+m-1}$ of multiplicity~$m$
\begin{equation}\label{eq:nc.h}
\eta_{j}=\frac{1}{(m+n-1-j)!}\frac{\partial^{m+n-1-j}}{\partial
z^{m+n-1-j}}
    \left.\left[\frac{(z-\la)^{m-1}}{s(z)c(z)}\right]\right|_{z=\la}
\end{equation}
for~$j=n,\dots, n+m-1$. And for
complex~$\la=\la_n=\la_{n+1}=...=\la_{n+m-1}=\overline{\la_{n+m}}=\overline{\la_{n+m+1}}=...=\overline{\la_{n+2m-1}}$
the element~$\eta_j$, $j=n,...n+m-1$, is defined
by~\eqref{eq:nc.h} and~$\eta_j=\overline{\eta_{j-m}}$
for~$j=n+m,n+m+1,...,n+2m-1$. Therefore,
\begin{equation}\label{eq:nc.res2r}
    \myres\limits_{z=\la}(\cL-z)^{-1}X=-\sum\limits_{k=n}^{n+m-1}\sum\limits_{l=n}^{n+m-1}h_{kl}(x,y_l)Y_k.
\end{equation}
for real eigenvalue~$\la=\la_n=...=\la_{n+m-1}$ and
\begin{equation}\label{eq:nc.res2c}
    \myres\limits_{z=\la}(\cL-z)^{-1}X=-\sum\limits_{k=n}^{n+2m-1}\sum\limits_{l=n}^{n+2m-1}h_{kl}(x,y_l)Y_k.
\end{equation}
for
complex~$\la=\la_n=\la_{n+1}=...=\la_{n+m-1}=\overline{\la_{n+m}}=\overline{\la_{n+m+1}}=...=\overline{\la_{n+2m-1}}$.
Now equating~\eqref{eq:nc.res1r} to~\eqref{eq:nc.res2r}
and~\eqref{eq:nc.res1c} to~\eqref{eq:nc.res2c} and taking linear
independence of~$Y_j$ (see Proposition~\ref{pr:A.spSAoPs}) into
account we derive the relation
\begin{equation}\label{eq:nc.rel}
\delta_{n+m-1}=\eta_{n+m-1}/\overline{\la},\quad
\delta_k=(\delta_k-\eta_{k+1})/\overline{\la}, \text{ for }
k=n,\dots, n+m-2.
\end{equation}
for eigenvalue~$\la=\la_n=\la_{n+1}=...=\la_{n+m-1}$ of
multiplicity~$m$.

%%%%%%%%%%%%%%%%%%%%%%%%%%%%%%%%%%%%%%%%%%%%%%%%%%%%%%%%%%%%%%%%%%
\subsection{Determining norming constants from two spectra}
%%%%%%%%%%%%%%%%%%%%%%%%%%%%%%%%%%%%%%%%%%%%%%%%%%%%%%%%%%%%%%%%%%

Let us now consider the problem~\eqref{eq:intr.spr} with so-called
mixed boundary conditions
\begin{equation}\label{eq:nc.mbc}
    y(0)=y^{[1]}(1)=0
\end{equation}
and denote by~$A_M$ the operator acting via
\[
    A_M y:=\ell(y)
\]
on the domain
\[
    \dom A_M:=\{y \in \dom \ell \mid  y(0)=y^{[1]}(1)=0\}.
\]
Define the operator pencil~$T_M$ by~\eqref{eq:pre.T} with~$A_M$
instead of~$A$. Then the spectral
problem~\eqref{eq:intr.spr},~\eqref{eq:nc.mbc} can be regarded as
that for~$T_M$.

We can analyse the pencil~$T_M$ in the same way as~$T$. Moreover,
by means of the operator~$A_M$ we can construct an energy
space~$\cE_M$, the corresponding Pontryagin space~$\Pi_M$ and
consider the corresponding linearization~$\cL_M$ therein as it was
done for~$T$. This will give that all the results of
sections~\ref{sec:SPOP} and~\ref{sec:Lin} concerning the
pencil~$T$ hold for~$T_M$.

Note that the function~$s(\la)$ is a characteristic function for
the problem~\eqref{eq:intr.spr}, \eqref{eq:pre.b.c.} and~$c(z)$ is
that for the problem~\eqref{eq:intr.spr},~\eqref{eq:nc.mbc}. This
means that some~$\la\in\bC$ is an eigenvalue of the pencil~$T$
($T_M$ respectively) of algebraic multiplicity~$m$ if and only if
it is a zero of~$s(z)$ (of~$c(z)$ respectively) of order~$m$. The
functions~~$s(z)$ and~$c(z)$ are of exponential type one and are
determined uniquely by their zeros by means of a canonical product
(see~\cite{You:2001}). Therefore, the spectra of~$T$ and~$T_M$
determine the functions~$s(z)$ and~$c(z)$ which by~\eqref{eq:nc.h}
and~\eqref{eq:nc.rel} determine the elements~$\delta_n$ of the
matrix~$D$ inverse to the Gramm matrix~$G$. Having~$D$ we
compute~$\beta_n$ and then the norming constants~$\al_n$ of~$T$
by~\eqref{eq:nc.ba}.

For instance, let the spectra of the pencils~$T$ and~$T_M$ consist
of simple eigenvalues~$(\la_n)_{n\in\bZ^*}$,~$\bZ^*:=\bZ\setminus
\{0\}$, and~$(\mu_n)_{n\in\bZ}$. Then for the eigenvalue~$\la_n$
of~$T$ we have (see~\eqref{eq:nc.h})
\[
\eta_n=\frac{1}{\dot{s}(\la_n)c(\la_n)}
\]
and, by definition,
\[
\delta_n=\frac{1}{\beta_n},
\]
where~$\beta_n$ is a norming constant of~$\cL$ corresponding
to~$\la_n$. These equalities together with~\eqref{eq:nc.rel} yield
\[
    \frac{\la_n}{\beta_n}=\frac{1}{\dot{s}(\la_n)c(\la_n)}
\]
which, by~\eqref{eq:nc.ba}, gives the equality for the norming
constant~$\al_n$ of~$T$ corresponding the eigenvalue~$\la_n$
\begin{equation}\label{eq:nc.al_rel}
    \al_n=\frac{\dot{s}(\la_n)c(\la_n)}{\la_n}.
\end{equation}

Therefore, the norming constant~$\al_n$ corresponding to the
eigenvalue~$\la_n$,~$n\in\bZ$ of~$T$ is determined by the formula
\begin{equation}\label{eq:cse.al}
    \al_n=\frac{1}{\la_n}{\prod\limits_{\substack{{k=-\infty}\\{k\ne
    0,n}}}^{\infty}}\frac{\la_n-\la_k}{\pi k}\prod\limits_{k=-\infty}^{\infty}\frac{\la_n-\mu_k}{\pi(k-
    1/2)}.
\end{equation}

 For the operator pencil~$T_M$ we can also define the norming
constants~$\al^M_{n}$ by~\eqref{eq:nc.nc_def} with~$T_M$ instead
of~$T$ and obtain for these norming constants analogous results as
for those of~$T$. In particular, the following theorem holds.
\begin{theorem}\label{thm:nc.pos_al_AM}
All the norming constants of~$T_M$ are positive if and only if the
operator~$A_M$ is positive.
\end{theorem}
For the pencil~$T_M$ and the corresponding linearization~$\cL_M$
we can define matrices~$G_M$,~$D_M$,~$H_M$ in the same way
as~$G$,~$D$,~$H$ were defined for~$T$ and obtain analogous
relations. It can be shown by direct analysis that for real
eigenvalue~$\mu=\mu_n=\dots=\mu_{n+m-1}$ of~$T_M$ of algebraic
multiplicity~$m$ the elements of the matrix~$H_M$ are
\[
    \eta^M_{j}=-\frac{1}{(m+n-1-j)!}\frac{\partial^{m+n-1-j}}{\partial
z^{m+n-1-j}}
    \left.\left[\frac{(z-\mu)^{m-1}}{s(z)c(z)}\right]\right|_{z=\mu}
\]
for~$j=n,\dots, n+m-1$. For
complex~$\mu=\mu_n=\mu_{n+1}=...=\mu_{n+m-1}=\overline{\mu_{n+m}}=\overline{\mu_{n+m+1}}=...=\overline{\mu_{n+2m-1}}$
the elements~$\eta^M_j$, $j=n,...n+m-1$, are defined by the same
formula and~$\eta^M_j=\overline{\eta^M_{j-m}}$
for~$j=n+m,n+m+1,...,n+2m-1$. Therefore, in the case of simple
eigenvalues we obtain the formulas for norming constants~$\al^M_n$
corresponding to the eigenvalues~$\mu_n$ of~$T_M$
\[
    \al^M_n=-\frac{s(\mu_n)\dot{c}(\mu_n)}{\mu_n}
\]
and so
\begin{equation}\label{eq:cse.alm}
    \al^M_n=-\frac{1}{\mu_n}{\prod\limits_{\substack{{k=-\infty}\\{k\ne
    0}}}^{\infty}}\frac{\mu_n-\la_k}{\pi k}\prod\limits_{\substack{{k=-\infty}\\{k\ne
    n}}}^{\infty}\frac{\mu_n-\mu_k}{\pi(k-
    1/2)}.
\end{equation}

%%%%%%%%%%%%%%%%%%%%%%%%%%%%%%%%%%%%%%%%%%%%%%%%%%%%%%%%%%%%%%%%%%%%%%%%%%%%%%%

\section{The case of real and simple eigenvalues}\label{sec:case}

%%%%%%%%%%%%%%%%%%%%%%%%%%%%%%%%%%%%%%%%%%%%%%%%%%%%%%%%%%%%%%%%%%%%%%%%%%%%%5

This section is devoted to the special case when the spectra
of~$T$ and~$T_M$ are real and simple. We shall establish some
conditions which guarantee that these spectra are real and simple.

We say that the spectra of the operator pencils~$T$ and~$T_M$
\emph{almost interlace} if they consist only of real and simple
eigenvalues, which can be labeled in increasing order as~$\la_n$,
${n\in\mathbb{Z}^*}$, and~$\mu_n$, ${n\in \mathbb{Z}}$,
respectively so that they satisfy the condition
\begin{equation}\label{eq:nc.aic}
    \mu_k<\la_k<\mu_{k+1}\; \text{for every } k\in\mathbb{Z}^*.
\end{equation}

\begin{theorem}
The following statements are equivalent
\begin{itemize}
    \item[(i)] The spectra of~$T$ and~$T_M$ almost interlace.
    \item[(ii)] A real number~$\mu_*$ exists such that the operator~$T_M(\mu_*)$ is negative.
\end{itemize}
\end{theorem}

\begin{proof}
((i)~$\Rightarrow$ (ii)) Let us firstly prove this implication for
the case when~$0\in(\mu_0,\mu_1)$. We shall show that then the
operator~$A_M=-T_M(0)$ is positive, which means that~$T_M(0)$ is
negative.

In view of~\eqref{eq:cse.alm} and the fact that the product
\[
    (\mu_n-\la_n)\prod\limits_{\substack{{k=-\infty}\\{k\ne
    0,n}}}^{\infty}\frac{(\mu_n-\la_k)}{\pi
    k}\frac{(\mu_n-\mu_k)}{\pi(k-1/2)}
\]
is negative, the sign of the norming constant~$\al^M_n$ is defined
by the sign of~$(\mu_n-\mu_0)/\mu_n$. Thus under our assumption
all the norming constants~$\al^M_n$ of~$T_M$ are positive. By
Theorem~\ref{thm:nc.pos_al_AM}, this gives positivity of~$A_M$.

 If~$0$ does not belong to~$(\mu_0,\mu_1)$ we take any point~$\mu_*$ from this
interval and shift the spectral parameter of~$T$ and~$T_M$
by~$\mu_*$ to obtain the pencils
\begin{align} \label{eq:crs.sT}
    \widehat{T}(\la)&=T(\la+\mu_*)=\la^2I-\la(B-2\mu_*I)+T(\mu_*)=\la^2I-2\la
    \widehat{B}-\widehat{A}\\
    \widehat{T}_M(\la)&=T_M(\la+\mu_*)=\la^2I-\la(B-2\mu_*I)+T_M(\mu_*)=\la^2I-2\la
    \widehat{B}-\widehat{A}_M\label{eq:crs.sTm}
\end{align}
with~$\widehat{B}:=B-2\mu_*I$,~$\widehat{A}:=-T(\mu_*)$
and~$\widehat{A}_M:=-T_M(\mu_*)$. Clearly, the spectra
of~$\widehat{T}$ and~$\widehat{T}_M$ almost interlace
with~$0\in(\mu_0,\mu_1)$. In view of the first part of this proof
the operator~$\widehat{A}_M$ is positive. Therefore, the
operator~$T_M(\mu_*)$ with this~$\mu_*$ is negative.

((ii)~$\Rightarrow$ (i)) Let the operator~$T_M(\mu_*)$ be
negative. Consider the operator pencil~$\widehat{T}_M$
of~\eqref{eq:crs.sTm} obtained from~$T_M$ by the shift of the
spectral parameter by~$\mu_*$. Then the
operator~$\widehat{A}_M=-T_M(\mu_*)$, and the operator $\widehat
A=-T(\mu_*)$ is also negative. By Theorem~\ref{thm:PS.SpT}, the
spectra of~$T$ and~$T_M$ are real and simple. The
eigenvalues~$\lambda_n$ and~$\mu_n$ can be enumerated so that
$\lambda_n =\pi n + p_0 + \tilde \la_n$
and~$\mu_n=\pi\left(n-\tfrac12\right) + p_0 + \tilde \mu_n$ with
$p_0:=\int_0^1 p(x)\,dx$ and $\ell_2$-sequences $(\tilde \mu_n),
(\tilde \lambda_n)$.

Next define the norming
constants~$\widehat{\al}^M_j$,~$j\in\bZ^*$, for~$\widehat{T}_M$.
In view of Theorem~\ref{thm:nc.pos_al_AM}, all these norming
constants  are positive and by~\eqref{eq:cse.alm} they are
determined by the formula
\[
\widehat{\al}^M_n=-\frac{1}{\mu_n-\mu_*}{\prod\limits_{\substack{{k=-\infty}\\{k\ne
    0}}}^{\infty}}\frac{\mu_n-\la_k}{\pi k}\prod\limits_{\substack{{k=-\infty}\\{k\ne
    n}}}^{\infty}\frac{\mu_n-\mu_k}{\pi(k-
    1/2)},
\]
where~$\la_n$,~$n\in\bZ^*$, and~$\mu_n$,~$n\in\bZ$, are the
eigenvalues of~$T$ and~$T_M$ respectively. Therefore, the
expression
\[
\frac{\widehat{\al}^M_{n+1}}{\widehat{\al}^M_n}=-\frac{\mu_{n}-\mu_*}{\mu_{n+1}-\mu_*}{\prod\limits_{\substack{{k=-\infty}\\{k\ne
    0}}}^{\infty}}\frac{\mu_{n+1}-\la_k}{\mu_n- \la_k}\prod\limits_{\substack{{k=-\infty}\\{k\ne
    n,n+1}}}^{\infty}\frac{\mu_{n+1}-\mu_k}{\mu_n-\mu_k}
\]
is positive. This gives that if~$\mu_n<\mu_*<\mu_{n+1}$, then
there is an even number of~$\la_k$ between~$\mu_n$ and~$\mu_{n+1}$
and otherwise there is an odd number of~$\la_k$ between~$\mu_n$
and~$\mu_{n+1}$. But the asymptotics of~$\mu_n$ and~$\la_n$
(see\cite{Pro:2012}) implies that the number of elements
of~$(\la_k)$ between~$\mu_n$ and~$\mu_{n+1}$ can not exceed~$1$
and that there is no~$\lambda_k$ between~$\mu_0$ and~$\mu_1$. This
gives that~$(\la_n)$ and~$(\mu_n)$ almost interlace
and~$\mu_*\in(\mu_0,\mu_1)$. This completes the proof.
\end{proof}

From the proof of the first implication in the above theorem we
immediately obtain the following corollaries.

\begin{corollary}\label{cor:nc.1}
If the spectra~$(\la_n)_{n\in\bZ^*}$ of~$T$
and~$(\mu_n)_{n\in\bZ}$ of~$T_M$ almost interlace, then for every
number~$\mu_*$ from the interval~$(\mu_0,\mu_1)$ the
operator~$T_M(\mu_*)$ is negative.
\end{corollary}

\begin{corollary}\label{cor:nc.2}
If for some~$\mu_*\in\bR$ the operator~$T_M(\mu_*)$ is negative,
then the spectra~$(\la_n)_{n\in\bZ^*}$ of~$T$
and~$(\mu_n)_{n\in\bZ}$ of~$T_M$ almost interlace
with~$\mu_*\in(\mu_0,\mu_1)$. Moreover, for every~$\mu$
from~$(\mu_0,\mu_1)$ the operator~$T_M(\mu)$ is negative .
\end{corollary}

%%%%%%%%%%%%%%%%%%%%%%%%%%%%%%%%%%%%%%%%%%%%%%%%%%%%%%%

\appendix

\section{Basics of Pontryagin spaces theory}

%%%%%%%%%%%%%%%%%%%%%%%%%%%%%%%%%%%%%%%%%%%%%%%%%%%%%%%%%%
In this appendix we recall some facts from the Pontryagin space
theory, which we use in the paper. The details of the theory, more
spectral properties of self-adjoint operators in Pontryagin spaces
and the proofs of the propositions given here can be found
in~\cite{Bog:1974,Lan:1982,AziIoh:1989}

A linear space~$\Pi$ is called an \emph{inner product space} if
there is a complex-valued function~$[\cdot,\cdot]$ defined
on~$\Pi\times \Pi$ so that the conditions
\begin{align*}
[\al_1 u_1+\al_2 u_2,v]&=\al_1[u_1,v]+\al_2[u_2,v]\\
[u,v]&=\overline{[v,u]}
\end{align*}
hold for every~$\al_1, \al_2\in \bC$ and~$u_1, u_2, u, v \in \Pi$.
The function~$[\cdot,\cdot]$ is then called an \emph{inner
product}. An inner product space~$(\Pi,[\cdot,\cdot])$ is a
\emph{Pontryagin space} of \emph{negative index}~$\kappa$ if~$\Pi$
can be written as
\begin{equation}\label{eq:PST.dec}
\Pi=\Pi_+\, [\dot{+}]\, \Pi_-,
\end{equation}
where~$[\dot{+}]$ denotes the direct~$[\cdot,\cdot]$-orthogonal
sum,~$(\Pi_{\pm},\pm[\cdot,\cdot])$ are Hilbert spaces and the
component~$\Pi_-$ is of finite dimension~$\kappa$.

An element~$x\in \Pi$ is said to be \emph{positive}
(\emph{negative, non-positive, non-negative, neutral} resp.)
if~$[x,x]>0$ ($[x,x]<0$, $[x,x]\le0$, $[x,x]\ge0$, $[x,x]=0$
resp.). A subspace~$\mathcal{M}$ of~$P$ is called~\emph{positive}
(\emph{negative, non-positive, non-negative, neutral} resp.) if
all its non-zero vectors are positive (negative, non-positive,
non-negative, neutral resp.)

In Pontryagin space of negative index~$\kappa$ the dimension of
any non-positive subspace can not exceed~$\kappa$. Moreover, a
non-positive subspace of Pontryagin space is maximal (i.e. such
that it is not properly included in any other non-positive
subspace) if and only if it is of dimension~$\kappa$.

Pontryagin spaces often arise from Hilbert spaces in the following
way. Suppose we have a Hilbert space~$(\mathcal{H},(\cdot,\cdot))$
and a bounded self-adjoint operator~$G$ in~$\mathcal{H}$
with~$0\in \rho(G)$ which has exactly~$\kappa$ negative
eigenvalues counted according to their multiplicities. Then with
an inner product
\[
    [x,y]:=(Gx,y), \; x,y\in \mathcal{H}
\]
the space~$(\mathcal{H},[\cdot,\cdot])$ is a Pontryagin space of
negative index~$\kappa$ for which the
decomposition~\eqref{eq:PST.dec} can be given with~$\Pi_+$
and~$\Pi_-$ being the spectral subspaces of~$G$ corresponding to
the positive and negative spectrum of~$G$ respectively.

Consider a Pontryagin space~$\Pi:=\bigl( \Pi,[\cdot,\cdot]\bigr)$
and a closed operator~$\mathcal{A}$ densely defined on~$\Pi$. An
\emph{adjoint}~$\mathcal{A}^{[\ast]}$ of~$\mathcal{A}$ in~$\Pi$ is
defined on the domain
\[
    \dom \mathcal{A}^{[\ast]}:=\{y\in \Pi\mid [\mathcal{A}\cdot,y] \text{ is a continuous linear functional on} \dom \mathcal{A}\}
\]
by the relation
\[
    [\mathcal{A}x,y]=[x, \mathcal{A}^{[\ast]}y],\; x\in \dom \mathcal{A},\; y\in\dom
    \mathcal{A}^{[\ast]}.
\]
The operator~$\mathcal{A}$ is \emph{symmetric}
if~$\mathcal{A}\subset \mathcal{A}^{[\ast]}$ and
\emph{self-adjoint} if~$\mathcal{A} = \mathcal{A}^{[\ast]}$. In
contrast to the case of Hilbert space, the spectrum of
self-adjoint operator in Pontryagin space is not necessarily real,
but it is always symmetric with respect to the real axis.

If for some eigenvalue~$\la_0$ of a self-adjoint operator in a
Pontryagin space all eigenvectors are positive (negative resp.)
then~$\la_0$ is called of \emph{positive} (\emph{negative} resp.)
\emph{type}.

\begin{proposition}\label{pr:A.spSAoPs}
Assume~$\cA$ is a self-adjoint operator in a Pontryagin
space~$\Pi$. Then
\begin{itemize}
   \item[(i)] The spectrum of~$\cA$ is real
with possible exception of at most~$\kappa$ pairs of
eigenvalues~$\la$ and~$\bar{\la}$ of finite multiplicities.
\item[(ii)]
If the spectrum of the operator~$\cA$ is discrete, then the set of
all eigenvectors and the corresponding associated vectors
of~$\mathcal{A}$ forms a basis in~$\Pi$.
\end{itemize}
\end{proposition}

 Denote
by~$\mathcal{M}_\la(\mathcal{A})$ the root space of the
operator~$\cA$ corresponding to eigenvalue~$\la$.
\begin{proposition}\label{pr:A.gram} Suppose~$\cA$ is a self-adjoint operator in a Pontryagin space. Then
\begin{enumerate}
\item For eigenvalues~$\la$ and~$\mu$ of~$\cA$ such that~$\la\ne \overline{\mu}$  the root spaces~$\mathcal{M}_{\la}(\mathcal{A})$
    and~$\mathcal{M}_{\mu}(\mathcal{A})$ are orthogonal;
\item The linear span of all the algebraic root spaces corresponding to the eigenvalues
of~$\cA$ in the upper (or lower) half plane is a neutral subspace
of~$\Pi$;
\item The root spaces~$\mathcal{M}_\la(\mathcal{A})$
    and~$\mathcal{M}_{\bar{\la}}(\mathcal{A})$  corresponding to
    complex conjugate eigenvalues~$\la$ and~$\bar{\la}$ are isomorphic. The spaces~$\mathcal{M}_\la(\mathcal{A})$
and~$\mathcal{M}_{\bar{\la}}(\mathcal{A})$ are of the same
dimension and have the same Jordan structure;
\item The length of a chain of eigen- and associated vectors does not exceed
$2\kappa+1$.
\end{enumerate}
\end{proposition}

%%%%%%%%%%%%%%%%%%%%%%%%%%%%%%%%%%%%%%%

\bigskip

\noindent\emph{Acknowledgement.} The author thanks her supervisor
Dr.~Rostyslav Hryniv for useful discussions and valuable
suggestions and for help with the preparation of the manuscript.

%%%%%%%%%%%%%%%%%%%%%%%%%%%%%%%%%%%%%%%%%%%

\bibliographystyle{abbrv}

%\bibliography{bibl_1}
\bibliography{My_bibl}
\end{document}